\documentclass[11pt]{amsart}

\usepackage[OT2,T1]{fontenc}
\usepackage[T1]{fontenc}
\usepackage[english]{babel}
\usepackage{csquotes}
\usepackage{lmodern}
\usepackage{textcomp}

\usepackage{pythonhighlight}
\usepackage{listings}

\definecolor{codegreen}{rgb}{0,0.6,0}
\definecolor{codegray}{rgb}{0.5,0.5,0.5}
\definecolor{codepurple}{rgb}{0.58,0,0.82}
\definecolor{backcolour}{rgb}{0.95,0.95,0.92}
\lstdefinestyle{mystyle}{
  backgroundcolor=\color{backcolour}, commentstyle=\color{codegreen},
  keywordstyle=\color{magenta},
  numberstyle=\tiny\color{codegray},
  stringstyle=\color{codepurple},
  basicstyle=\ttfamily\footnotesize,
  breakatwhitespace=false,         
  breaklines=true,                 
  captionpos=b,                    
  keepspaces=true,                 
  numbers=left,                    
  numbersep=5pt,                  
  showspaces=false,                
  showstringspaces=false,
  showtabs=false,                  
  tabsize=2
}
\usepackage{amssymb, amsmath, amsthm}
\usepackage{fullpage}
\usepackage{graphicx}
\usepackage{svg}
\usepackage{calc}

\usepackage{caption}
\usepackage{subcaption}

\usepackage[bottom]{footmisc}

\usepackage{hyperref}
\usepackage{enumerate}

\usepackage{soul}

\usepackage{amsmath}
\usepackage{amssymb}
\usepackage{amsthm}
\usepackage{stmaryrd}

\usepackage{graphicx} 

\usepackage{mathtools}

\usepackage{mathabx,epsfig}

\usepackage{stmaryrd}

\usepackage{wasysym}

\newlength\tindent
\usepackage{color}
\definecolor{darkred}{RGB}{170,0,0}
\definecolor{darkgreen}{RGB}{0,120,0}

\usepackage{soul}
\setuldepth{noDescent}

\usepackage{tikz}
\usetikzlibrary{cd}
\usetikzlibrary{babel}
\tikzset{
	labl/.style={anchor=south, rotate=90, inner sep=.5mm}
}

\usepackage{tcolorbox}

\usepackage{hyperref}
\hypersetup{
	colorlinks,
	citecolor=blue,
	filecolor=black,
	linkcolor=black,
	urlcolor=black
}

\usepackage{booktabs}

\usepackage{nowidow}

\usepackage{calligra}
\usepackage{mathrsfs}

\usepackage{enumerate}

\usepackage[retainorgcmds]{IEEEtrantools}

\usepackage{parskip}

\definecolor{lightgray}{rgb}{0.4,0.4,0.4}

\let\le\leqslant
\let\ge\geqslant

\DeclareMathOperator{\Gal}{Gal}

\DeclareMathOperator{\PGL}{PGL}

\DeclareMathOperator{\sHom}{\mathscr{H}\text{\kern -3pt {\calligra\large om}}\,}

\mathcode`l="8000
\begingroup
\makeatletter
\lccode`\~=`\l
\DeclareMathSymbol{\lsb@l}{\mathalpha}{letters}{`l}
\lowercase{\gdef~{\ifnum\the\mathgroup=\m@ne \ell \else \lsb@l \fi}}%
\endgroup

\renewcommand{\phi}{\varphi}
\renewcommand{\epsilon}{\varepsilon}

\usepackage{mathtools}

\newcommand{\BF}{{\mathbb{F}}}

\newcommand{\BP}{{\mathbb{P}}}
\newcommand{\BQ}{{\mathbb{Q}}}

\newcommand{\CO}{{\mathcal O}}

\newcommand{\CX}{{\mathcal X}}
\newcommand{\CY}{{\mathcal Y}}

\newtheorem{Prin}{Principle}
\newtheorem{Prop}[Prin]{Proposition}

\newtheorem{Satz}[Prin]{Theorem}

\newcommand{\thistheoremname}{}
\newtheorem*{genericthm}{\thistheoremname}

\theoremstyle{definition}

\definecolor{ferrarired}{rgb}{1.0, 0.11, 0.0}
\newtcolorbox{01Def}{colback=black!0,colframe=ferrarired!50}

\usepackage[style=alphabetic]{biblatex} 
\begin{document}

\title{Computing the semistable reduction of a particular plane quartic curve at $p=3$}

\begin{abstract}
We explain how to determine the semistable reduction of a particular plane quartic curve at $p=3$ that appears in the attempts of Rouse, Sutherland, and Zureick-Brown to compute the rational points on the non-split Cartan modular curve $X^+_\textrm{ns}(27)$.
\end{abstract}

\author{Ole Ossen}

\email{ole.ossen@uni-ulm.de}

\maketitle

\section*{Introduction}

Consider the local field $K\coloneqq\BQ_3(\zeta_3)$. We normalize its valuation $v_K\colon K\to\BQ\cup\{\infty\}$ so that $v_K(3)=1$, that is, so that $v_K(\zeta_3-1)=1/2$. Our goal is to determine the semistable reduction of the plane quartic curve $Y\subset\BP_K^2$ given by the equation
\begin{equation}
\label{eq-quartic-as-given}
\begin{aligned}
x^4 &+ (\zeta_3 - 1)x^3y + (3\zeta_3 + 2)x^3z - 3x^2z^2 + (2\zeta_3 + 2)xy^3 - 3\zeta_3 xy^2z&\\
&+ 3\zeta_3 xyz^2 - 2\zeta_3 xz^3 - \zeta_3 y^3z + 3\zeta_3 y^2z^2 + (-\zeta_3 + 1)yz^3 + (\zeta_3 + 1)z^4=0.
\end{aligned}
\end{equation}
This curve appears as a quotient of $X^+_\textrm{ns}(27)$, the non-split Cartan modular curve of level $27$, in Section 9 of \cite{rousesutherlandzureickbrown}, where attempts to compute the rational points on $X^+_\textrm{ns}(27)$ via $K$-rational points on $Y$ are described. Computing the latter might be feasible using the methods developed in \cite{balakrishnandogramuellertuitmanvonk}, but this requires at the least knowledge of the special fiber of a regular semistable model of $Y$. We have the following result:

\begin{Satz}
\label{theresult}
There exists a field extension $L/K$ of degree $108=2^2\cdot3^3$ (explicitly described below), of ramification index $54=2\cdot3^3$ and residue field $\BF_{3^2}$, with the following properties:
\begin{enumerate}[(a)]
\item The curve $Y_L$ has semistable reduction. The special fiber of a stable model of $Y_L$ consists of three components $\overline{Y}_1,\overline{Y}_2,\overline{Y}_3$ of genus $1$, each a smooth plane curve given by the equation
\begin{equation}\label{eq-char3equation}
y^3-y=x^2\quad\text{over $\BF_{3^2}$},
\end{equation}
and of one rational component. The components are configured as in Figure \ref{fig-1}.
\item The special fiber of a regular semistable model of $Y_L$ consists of three components $\overline{Y}_1,\overline{Y}_2,\overline{Y}_3$ of genus $1$, each a smooth plane curve given by Equation \eqref{eq-char3equation}, and of $25$ rational components, configured as in Figure \ref{fig-2}.
\end{enumerate}
\end{Satz}

\begin{figure}[!htb]
\centering
\begin{minipage}{.5\textwidth}
\centering
\includegraphics[scale=0.2]{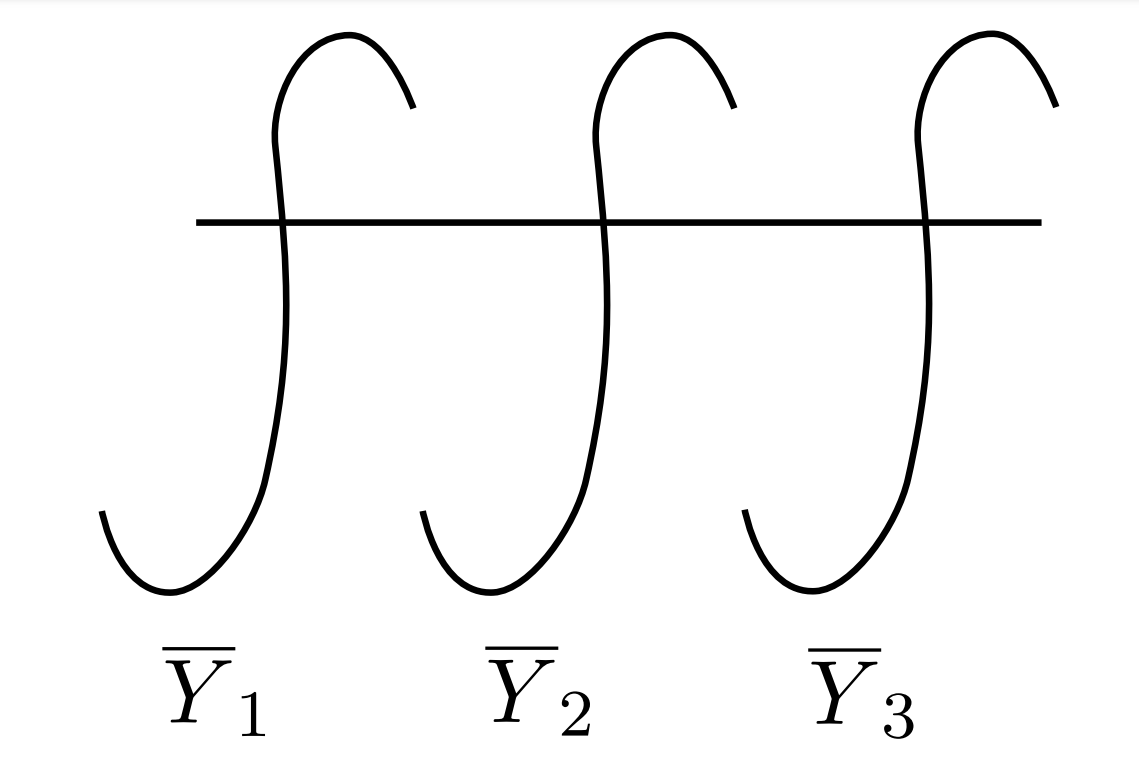}
\caption{The special fiber of the stable model of $Y_L$}
\label{fig-1}
\end{minipage}%
\begin{minipage}{.5\textwidth}
\centering
\includegraphics[scale=0.25]{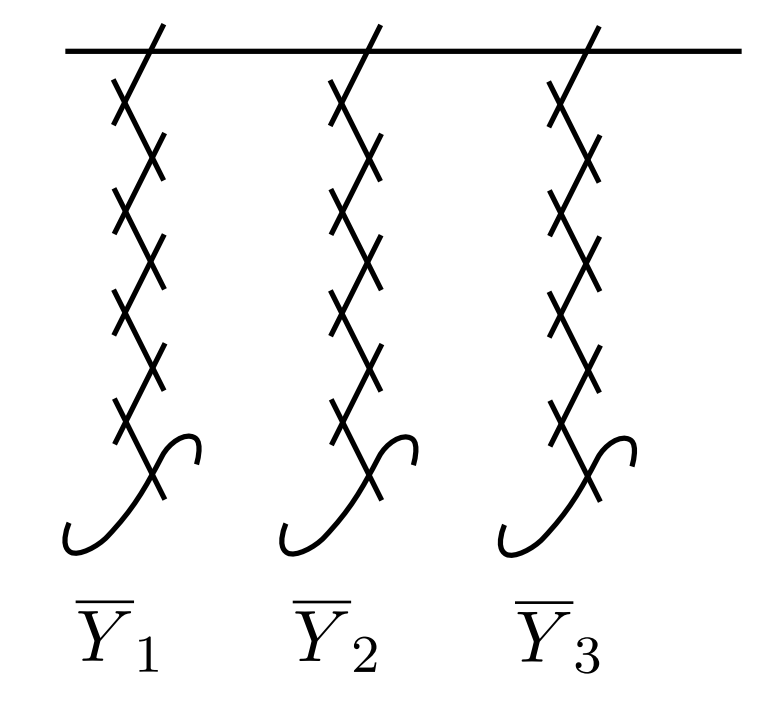}
\captionof{figure}{The special fiber of a regular semistable model of $Y_L$}
\label{fig-2}
\end{minipage}
\end{figure}

After recalling some preliminaries in Section 1, we give two proofs of Theorem \ref{theresult}, each based on a different degree $3$ cover of the projective line by $Y$. An appendix contains several computer programs used to arrive at the results of Theorem \ref{theresult}.

While we describe certain models of $Y$ and check that they have the reduction claimed in Theorem \nolinebreak \ref{theresult}, the question of how to find these models is left unaddressed. Developing a systematic method for finding semistable models of plane quartics at $p=3$ is part of the author's PhD thesis. Here, the method is left a black box, and only its output is considered. However, given the output, it is not hard to check that the corresponding model does in fact have the claimed reduction, and we explain how to do this in detail.

\textbf{Acknowledgement}: I want to thank my doctoral supervisor Stefan Wewers for many helpful discussions, and for commenting on an earlier version of this document.

\section{Models and valuations}

In this section, $K$ is a field that is complete with respect to a discrete valuation $v_K\colon K\to\BQ\cup\{\infty\}$, with ring of integers $\CO_K$ and residue field $k$. We also denote the extension of $v_K$ to an algebraic closure $\overline{K}$ of $K$ by $v_K$. Let $Y$ be a smooth irreducible projective algebraic curve over $K$. A \emph{model} of $Y$ is a normal flat and proper $\CO_K$-scheme $\CY$ with an isomorphism $\CY\otimes_{\CO_K}K\cong Y$.

A valuation $v$ on the function field $K(Y)$ is called a \emph{Type II valuation} if it extends $v_K$ and has residue field of transcendence degree $1$ over $K$. Given a model $\CY$ of $Y$ and an irreducible component $Z$ of the special fiber $\CY_s=\CY\otimes_{\CO_K}k$ of $\CY$, the local ring $\CO_{\CY,\xi_{Z}}$ at the generic point $\xi_{Z}$ of $Z$ is a discrete valuation ring. Denote its valuation by $v_Z$; it is a Type II valuation.

\begin{Prop}
\label{julian-prop}
\begin{enumerate}[(a)]
\item The map
\begin{equation*}
\CY\mapsto\{v_Z\mid Z\subseteq\CY_s\text{ irreducible component}\}
\end{equation*}
induces a bijection between isomorphism classes of models of $Y$ and finite non-empty sets of Type II valuations on $K(Y)$.
\item Let $Y\to X$ be a cover of smooth irreducible projective algebraic curves over $K$. Let $\CX$ be a model of $X$ and let $\CY$ be the normalization of $\CX$ in the function field $K(Y)$. Then $\CY$ is a model of $Y$ and the valuations corresponding to $\CY$ via the bijection in (a) are the extensions of valuations corresponding to $\CX$.
\end{enumerate}
\end{Prop}
\begin{proof}
See \cite[Chapter 3 and Section 5.1.2]{rueth}.
\end{proof}

Given a model $\CY$, we will often talk of the valuation corresponding to an irreducible component of $\CY_s$ via Proposition \ref{julian-prop}, or of the irreducible component corresponding to a valuation on $K(Y)$.

Valuations on the rational function field $K(x)$ can be described using \emph{discoids}. Given a monic irreducible polynomial $\psi\in K[x]$ and a rational number $\lambda\in\BQ$, the set 
\begin{equation*}
D(\psi,\lambda)\coloneqq\{x\in\overline{K}\mid v_K(\psi(x))\ge\lambda\}
\end{equation*}
is called the \emph{discoid with center $\psi$ and radius $\lambda$}. Given a discoid $D$, there is a valuation $v_D$ on $K(x)$ determined by 
\begin{equation*}
v_D(f)=\inf\{v_K(f(x))\mid x\in D\},\qquad f\in K[x].
\end{equation*}
It is shown in \cite[Theorem 4.56]{rueth} that every Type II valuation on $K(x)$ arises in this way. If $\psi=x-\alpha$ is linear, then $D$ is simply a closed disk in $\overline{K}$ and $v_D$ is the \emph{Gauss valuation} with center $\alpha$ and radius $\lambda$.

Note that radii of disks and discoids are defined in terms of valuations instead of in terms of the corresponding absolute values. In particular, the larger the radius of a disk or discoid is, the smaller it is.

Berkovich (\cite{berkovich}) has defined an analytification $Y^\textrm{an}$ of $Y$. Its underlying set contains among others the closed points of $Y$ and so-called Type II points, corresponding to the Type II valuations introduced above. We will sometimes use this geometric perspective and talk of \emph{skeletons} of $Y$. These are finite metric subgraphs contained in the analytification $Y^\textrm{an}$ capturing the reduction type of $Y$, see \cite[Chapter 4]{berkovich}.

Let $\phi\colon Y\to X$ be a morphism of smooth irreducible projective curves and suppose that we are given a model $\CX$ of $X$. Given an irreducible component $Z$ of $\CX_s$, the function field of $Z$ is by construction the residue field of the valuation $v_Z$ corresponding to $Z$. Given an extension $w$ of $v_Z$ to $K(Y)$, $w$ corresponds by Proposition \ref{julian-prop}(b) to an irreducible component $W$ of $\CY_s$, where $\CY$ is the normalization of $\CX$ in $K(Y)$. The extension of the residue fields of $w$ and $v_Z$ is also the extension of function fields corresponding to the morphism $W\to Z$.

Suppose now that $\phi$ is of degree $p$, where $p$ is the characteristic of $k$. If the extension of residue fields of $w$ and $v_Z$ is of degree $p$, we say that $v_Z$ (or the corresponding point of $X^\mathrm{an}$) is a \emph{wild topological branch point} of $\phi$. In \cite{ctt}, a \emph{different function} $\delta$ is used to describe the wild topological branch locus of a cover of analytic curves, see \cite[Lemma 4.2.2 and Remark 4.2.3]{ctt} in particular. We use it in the appendix to determine wild topological branch loci in two cases.

\section{Using boundary points of the wild topological branch locus}

In this section we give a first proof of Theorem \ref{theresult}. We begin by slightly rewriting Equation \eqref{eq-quartic-as-given} defining the curve $Y$. Plugging in $z+x(2\zeta_3+2)/\zeta_3$ for $z$ eliminates the $xy^3$ term in \eqref{eq-quartic-as-given}. Then setting $z=1$ we arrive at the affine equation
\begin{equation}
\label{eq-Aversion}
y^3+Ay^2+By+C=0,
\end{equation}
where
\begin{equation*}
A=(6\zeta_3 + 12)x^2+ (36\zeta_3 + 9)x- 27,
\end{equation*}
\begin{equation*}
B=(9\zeta_3 - 18)x^3+ (-108\zeta_3 - 108)x^2+ (-162\zeta_3 + 81)x+ (81\zeta_3 + 162),
\end{equation*}
\begin{equation*}
C=(27\zeta_3 - 243)x^4+ (-1458\zeta_3 - 999)x^3+ (-1215\zeta_3 + 1701)x^2+ (1944\zeta_3 + 2430)x+ 729\zeta_3.
\end{equation*}
In this way we obtain a degree-$3$ cover $\phi\colon Y\to\BP_K^1$ corresponding to the extension of the rational function field $K(x)$ generated by $y$.

A certain discoid is crucial for describing the semistable reduction of $Y$, namely the discoid
\begin{equation}
\label{eq-crucial-discoid}
D(\psi,45/4)=\{x\in\overline{K}\mid v_K(\psi(x))\ge45/4\},
\end{equation}
where
\begin{equation}\label{eq-phi}
\begin{aligned}
\psi&=x^9 + (9\zeta_3 - 9)x^8 + (54\zeta_3 + 27)x^7 + (54\zeta_3 - 27/2)x^6 + (243\zeta_3 + 972)x^5 + 729\zeta_3 x^4\\ &+ (2916\zeta_3 - 1458)x^3 + (37179\zeta_3 + 41553)x^2 + (6561\zeta_3 + 6561/8)x - 63423\zeta_3 + 155277.
\end{aligned}
\end{equation}
It is by no means obvious how to find this discoid. As mentioned in the introduction, it comes from a general method for computing the semistable reduction of plane quartics at $p=3$ that is part of the author's PhD thesis.

In a splitting field of $\psi$, the roots of $\psi$ split into three clusters of three roots each. More precisely, we have the following: For any root $\alpha$ of $\psi$, the smallest closed disk around $\alpha$ containing another root of $\psi$ is of radius $3/2$ and contains three roots in total. The smallest closed disk containing all nine roots of $\psi$ is of radius $7/6$.

In Figure \ref{fig-6}, a skeleton of $\BP^1_{\overline{K}}$ separating the zeros of $\psi$ is pictured. The dashed lines represent disks of the indicated radius, centered at one of the roots of $\psi$. The intersection of the wild topological branch locus of $\phi$ with this skeleton consists of the intervals in red (endpoints included). See the discussion before Code Listing \ref{code-4} in the appendix for how to find this wild topological branch locus.

\begin{figure}[!htb]
\centering
\includegraphics[scale=0.44]{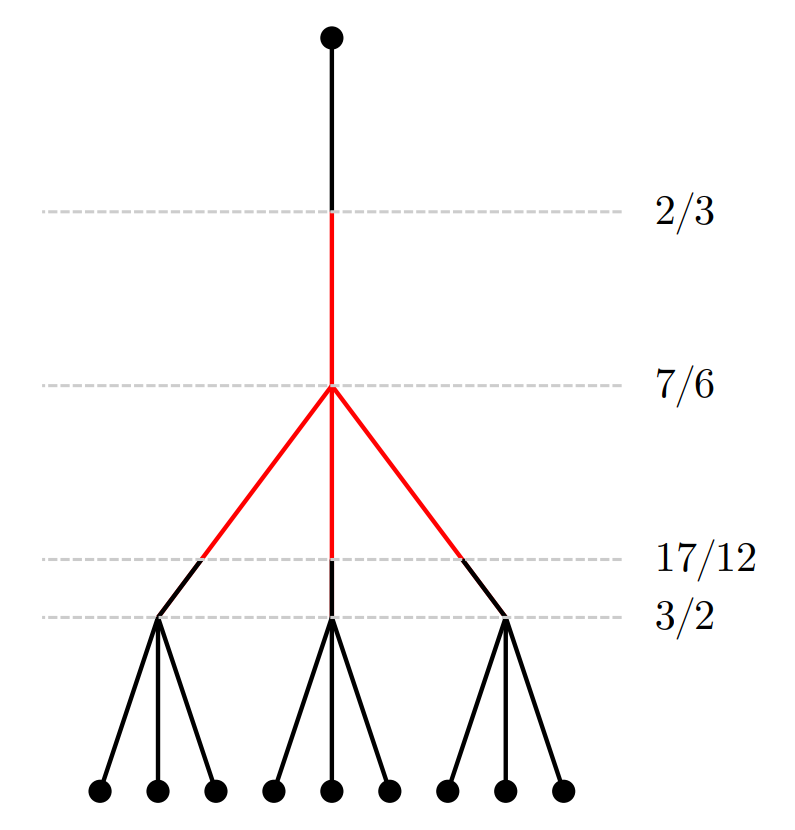}
\caption{A skeleton of $\BP^1_{\overline{K}}$ containing the zeros of $\psi$}
\label{fig-6}
\end{figure}

We now begin the proof of Theorem \ref{theresult}. We claim that $Y$ has semistable reduction over an extension $L/K$ satisfying the following:
\begin{itemize}
\item The value group $v_L(L^\times)$ contains the radius $17/12$
\item For each of the three closed disks of radius $17/12$ in Figure \ref{fig-6}, there is an $L$-rational point on $Y$ that maps to this disk
\end{itemize}
For example, we could first adjoin to $K$ a splitting field of $\psi$, yielding rational points in each of the closed disks of radius $17/12$ in Figure \ref{fig-6}, and then adjoin the $y$-coordinate of a rational point on $Y$ above each of these rational points. Finally, take an extension for which $17/12$ is in the value group $v_L(L^\times)$.

However, it turns out that this extension is larger than necessary. In Code Listing \ref{code-5}, we show that there exists an extension $L/K$ satisfying both conditions above and having the degree and ramification index claimed in Theorem \ref{theresult}.

It remains to show that over an extension $L$ satisfying the conditions above, $Y$ has semistable reduction with reduction as in Figure \ref{fig-1}. To this end, use Proposition \ref{julian-prop}(a) to construct a model $\CX$ of $\BP^1_L$ whose special fiber has four components, corresponding to the boundary point of the disk of radius $7/6$ and to the boundary points of the three disks of radius $17/12$ in Figure \ref{fig-6}. Denote the three components of $\CX_s$ corresponding to the latter by $Z_1,Z_2,Z_3$. 

Let $\CY$ denote the normalization of $\CX$ in the function field $K(Y_L)$. We claim that $\CY$ has semistable reduction. The calculation in Code Listing \ref{code-4} shows that $\CY_{\CO_{\overline{K}}}$ has semistable reduction. Indeed, it computes the component of $(\CY_{\CO_{\overline{K}}})_s$ above the component of $(\CX_{\CO_{\overline{K}}})_s$ corresponding to one of the disks of radius $17/12$ in Figure \ref{fig-6}. It is the genus-$1$ curve given by Equation \eqref{eq-char3equation}. Because of the symmetry afforded by the action of $\Gal(\overline{K}/K)$, all three components of $(\CY_{\CO_{\overline{K}}})_s$ above the components corresponding to the disks of radius $17/12$ are genus-$1$ curves given by \eqref{eq-char3equation}. Thus $(\CY_{\CO_{\overline{K}}})_s$ is a curve with four components; three components are of genus $1$ and intersect the last component, which is rational, in one point each. Since the arithmetic genus of $(\CY_{\CO_{\overline{K}}})_s$ is $3$, it follows from \cite[Proposition 7.5.4]{liu} that $(\CY_{\CO_{\overline{K}}})_s$ is a semistable curve as pictured in Figure \ref{fig-1}.

Now let $M/L$ be a minimal extension over which $Y$ has semistable reduction. Since there is an $L$-rational point in each of the closed disks of radius $17/12$ in Figure \ref{fig-6} and $17/12\in v_L(L^\times)$, the action of $\Gal(M/L)$ on the components $Z_i$ of $\CX_s$ is trivial. The only automorphisms of the component $\overline{Y}_i$ over $Z_i$ (the Artin Schreier cover given by \eqref{eq-char3equation}) are given by $(x,y)\mapsto(x,y+c)$, $c=0,1,2$. Since $Y$ has an $L$-rational point above each of the closed disks of radius $17/12$, we see that $\Gal(M/L)$ acts trivially on $\CY_s$. It follows from \cite[Theorem 10.4.44]{liu} that $M=L$.

This concludes the proof of Theorem \ref{theresult}(a). To prove (b), consider the map $\CY_s\to\CX_s$ induced by $\phi$, sketched in Figure \ref{fig-7}. The inseparable component, corresponding to the disk of radius $7/6$ in Figure \ref{fig-6}, is in red. Each double point has its \emph{thickness} written next to it (see \cite[Definition 10.3.23]{liu}). It equals the width of the anulus on the analytification which is the inverse image of the double point under the reduction map.

For the double points on $\CX_s$, this thickness $7/6-17/12=1/4$ can be read off from Figure \ref{fig-6}. Since $\phi$ is degree-$3$ on these anuli, we get the thickness $1/12$ for the double points on $\CY_s$ (see for example \cite[Lemma 3.5.8]{ctt}). Part (b) of Theorem \ref{theresult} follows, since $L/\BQ_3$ is of ramification index $108$: An anulus of width $1/12$ splits into nine anuli of width $1/108$.

\begin{figure}[!htb]
\centering
\includegraphics[scale=0.4]{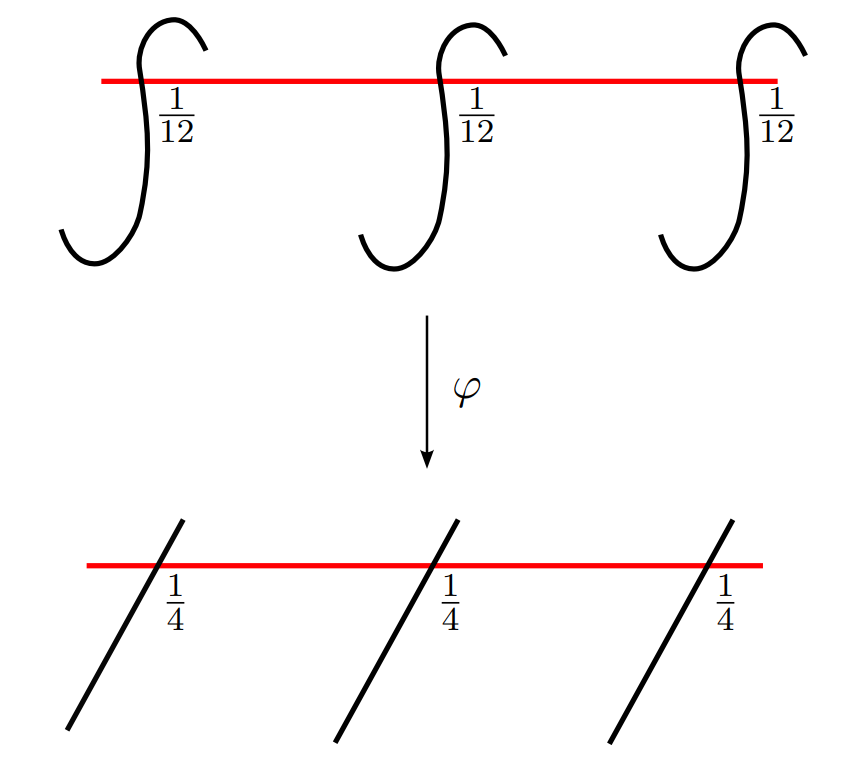}
\caption{The cover $\CY_s\to\CX_s$}
\label{fig-7}
\end{figure}

\section{Using separation of branch points}

In this section we use a different strategy to prove Theorem \ref{theresult}. After enlarging the base field to a finite extension $K'/K$, we change coordinates to obtain a different cover $\phi\colon Y\to\BP_{K'}^1$ from the one in the previous section. It turns out that (after a further finite extension $L/K'$) for a model $\CX$ of $\BP_L^1$ that ``separates the branch points'' of $\phi$, the normalization $\CY$ of $\CX$ in the function field $K(Y_L)$ is a semistable model of $Y_L$. 

It follows that the induced map $\CY_s\to\CX_s$ is an \emph{admissible cover} (\cite[Section 4]{harrismumford}), unlike the induced map on special fibers from the previous section. It is not clear if one can always find a cover $Y\to\BP^1_{K'}$ for which ``separation of branch points'' works. While the crucial discoid \eqref{eq-crucial-discoid} from the previous section was obtained from a general method, the coordinate changes below are ad-hoc and were essentially found by accident.

The curve defined by \eqref{eq-quartic-as-given} has an inflection point $[a^2:a:1]$, where $a$ satisfies the equation $a^3+\zeta_3+1=0$. A third root of unity $\zeta_3$ and an element $a$ satisfying this equation are both contained in $K'\coloneqq\BQ_3(\zeta_9)$; in fact, we can take $\zeta_3=\zeta_9^3$ and  $a=\zeta_9^2$.

We now transform \eqref{eq-quartic-as-given} by applying coordinate transformations achieving the following:
\begin{itemize}
\item Move the inflection point $[a^2:a:1]$ to $[0:0:1]$
\item Move the tangent line to $Y_{K'}$
at $[0:0:1]$ to the line $\{[s:0:t]\mid s,t\in K'\}$
\item In the homogeneous polynomial defining $Y_{K'}$, eliminate the terms of degree $2$ in $z$
\item Replace $x$ with $9z$, $y$ with $3x$, and $z$ with $y$
\end{itemize}
The last step is of course not important, but will allow us to use the variable names $x$ and $y$ in the same way as in the previous section (as generator of the rational function field and of the function field of $Y$ respectively). Each of the above bullet points corresponds to one of the matrices in the following product

\begin{equation*}
\begin{pmatrix}
0&3&0\\0&0&1\\9&0&0
\end{pmatrix}\begin{pmatrix}
9&0&-5\zeta_9^5 + \zeta_9^4 + 4\zeta_9^3 - \zeta_9^2 - \zeta_9 + 2\\0&9&3\zeta_9^4-3\zeta_9^3+6\zeta_9+3\\0&0&9
\end{pmatrix}\begin{pmatrix}
1&0&0\\-\zeta_9^4&1&0\\0&0&1
\end{pmatrix}\begin{pmatrix}1&0&0\\0&1&0\\\zeta_9^4&\zeta_9^2&1\end{pmatrix}\in\PGL_3(K').
\end{equation*}
Applying this transformation to \eqref{eq-quartic-as-given} and setting $z=1$, we arrive at the affine equation
\begin{equation}\label{eq-simpleform}
y^3+By+C=0,
\end{equation}
where
\begin{equation*}
\begin{aligned}
B&=(\zeta_9^5 + \zeta_9^3 + \zeta_9^2)x^3\\ &+ 9\zeta_9^4x^2\\ &+ (9\zeta_9^5 - 27\zeta_9^4 - 18\zeta_9^3 + 18\zeta_9^2 - 36)x\\ &+ 9\zeta_9^5 + 54\zeta_9^4 + 18\zeta_9^3 - 36\zeta_9^2 + 27\zeta_9 + 63,
\end{aligned}
\end{equation*}
\begin{equation*}
\begin{aligned}
C&=(-2\zeta_9^5 + \zeta_9^4 + \zeta_9^3 - \zeta_9^2 + 2\zeta_9 - 1)x^4\\ &+ (24\zeta_9^5 - 6\zeta_9^4 + 15\zeta_9^3 + 12\zeta_9^2 + 15)x^3\\ &+ (-45\zeta_9^5 + 9\zeta_9^4 - 18\zeta_9^3 + 45\zeta_9^2 - 63\zeta_9 - 36)x^2\\ &+ (135\zeta_9^5 + 135\zeta_9^4 - 135\zeta_9^3 - 54\zeta_9^2 + 270\zeta_9 - 27)x\\ &- 189\zeta_9^5 - 81\zeta_9^4 + 279\zeta_9^3 - 108\zeta_9^2 - 270\zeta_9 + 72.
\end{aligned}
\end{equation*}
Again, we obtain a degree-$3$ cover $\phi\colon Y\to\BP^1_{K'}$, corresponding to the extension of the rational function field $K'(x)$ generated by $y$.

To compute the semistable reduction of $Y_{K'}$, we begin with the discriminant of the polynomial $y^3+By+C$, 
\begin{equation}\label{eq-discriminant}
\Delta=-4B^4-27C^2\in K'[x].
\end{equation}
It is an irreducible polynomial, whose splitting field $K''$ is of degree $18$ over $K'$, of ramification index $9$. The roots of $\Delta$ behave similarly to the roots of the polynomial $\psi$ from the previous section, clustering into three groups of three roots each: For any root $\alpha$ of $\Delta$, the smallest closed disk around $\alpha$ containing another root of $\Delta$ is of radius $4/3$ and contains three roots in total. The smallest closed disk containing all nine roots of $\Delta$ is of radius $10/9$.

In Figure \ref{fig-4}, the skeleton of $\BP^1_{\overline{K}}$ spanned by the ten branch points of $\phi$ is pictured. The nine zeros of $\Delta$ are at the bottom, the point at infinity is at the top. The dashed lines represent disks of the indicated radius, centered at one of the zeros of $\Delta$. The intersection of the wild topological branch locus of $\phi$ with this skeleton consists of the intervals in red (endpoints included). See the appendix and Code Listing \ref{code-1} in particular for how to find this wild topological branch locus.

\begin{figure}[!htb]
\centering
\includegraphics[scale=0.5]{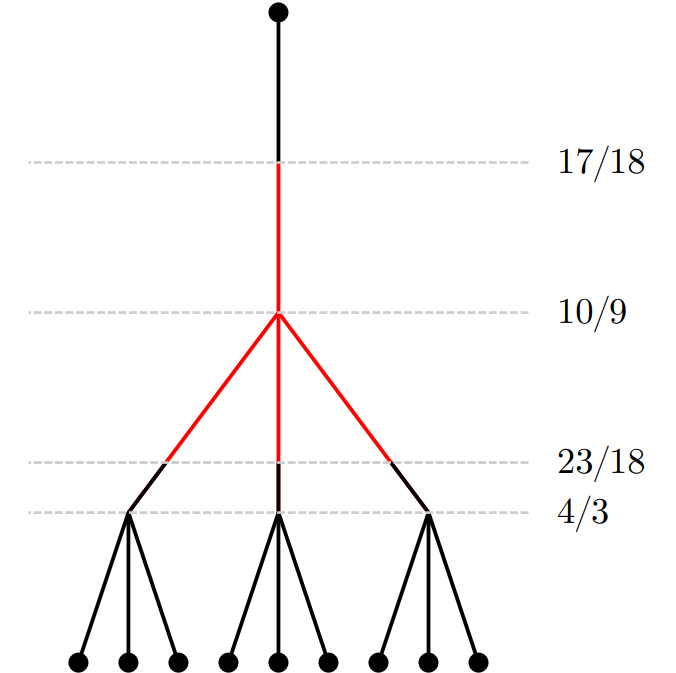}
\caption{The skeleton of $\BP_{\overline{K}}^1$ spanned by the branch points of $\phi$, with the wild topological branch locus marked in red}
\label{fig-4}
\end{figure}

Using Proposition \ref{julian-prop}(a), construct a model $\CX$ of $\BP^1_{K''}$ whose special fiber has seven components corresponding to the following Type II points:
\begin{itemize}
\item The boundary point of the disk of radius $10/9$ in Figure \ref{fig-4}
\item The boundary points of the three disks of radius $23/18$ in Figure \ref{fig-4}
\item The boundary points of the three disks of radius $4/3$ in Figure \ref{fig-4}
\end{itemize}

Let $L$ be the compositum of $K''$ and a ramified extension $K'''/K'$ of degree $2$. Then $L$ is of degree 
\begin{equation*}
[K'':K']\cdot[K''':K']\cdot[K':K]=18\cdot2\cdot3=108,
\end{equation*}
over $K$, of ramification index $9\cdot2\cdot3=54$, as in the statement of Theorem \ref{theresult}. Moreover, $Y_L$ has semistable reduction. To see this, consider a valuation $v$ corresponding to one of the disks of radius $4/3$ in Figure \ref{fig-4}. Since the corresponding Type II point does not lie in the wild topological branch locus, $v$ splits in $K(Y_L)$. Since $K'''/K'$ is chosen to be of ramification index divisible by $2$, the valuation $v$ does not ramify in $K(Y_L)$ by Abhyankar's Lemma. Rather, there is a valuation $w$ above $v$ with residue field extension of degree $2$ (and another valuation above $v$ with residue field extension of degree $1$).

In the appendix, Code Listing \ref{code-2} explicitly computes this function field. It is given by the affine equation over $\BF_{3^2}$
\begin{equation}\label{eq-completelydifferent}
y^2=x^3+x-1.
\end{equation}
Note that this genus-$1$ curve is completely different from \eqref{eq-char3equation} if considered as a cover of $\BP^1_{\BF_{3^6}}$, with function field $\BF_{3^2}(x)$, but is isomorphic to the curve defined by \eqref{eq-char3equation} as a plane curve (switch $x$ and $y$). The curve defined by \eqref{eq-completelydifferent} is branched at four points, namely infinity and specializations of the three roots of $\Delta$ in the disk corresponding to $v$.

Because of the symmetry afforded by the action of $\Gal(\overline{K}/K')$, we get a genus-$1$ component given by \eqref{eq-completelydifferent} for each of the three disks of radius $4/3$ in Figure \ref{fig-4}. We conclude that $\CY_s$ is semistable in the same way as in the previous section: The special fiber $\CY_s$ has arithmetic genus $3$, so can only have double points as singularities and no further components of positive genus.

In Figure \ref{fig-5}, the induced map $\CY_s\to\CX_s$ is pictured. The inseparable component, corresponding to the disk of radius $10/9$ in Figure \ref{fig-4}, is in red. The relevant double points have their thickness written next to them (compare with the end of the previous section). For the double points on $\CX_s$, this can be read off from Figure \ref{fig-4}. The double points on $\CY_s$ follow, since $\phi$ induces a degree-$2$ cover over the anuli of width $1/18$ and a degree-$3$ cover over the anuli of width $1/6$.

\begin{figure}[!htb]
\centering
\includegraphics[scale=0.5]{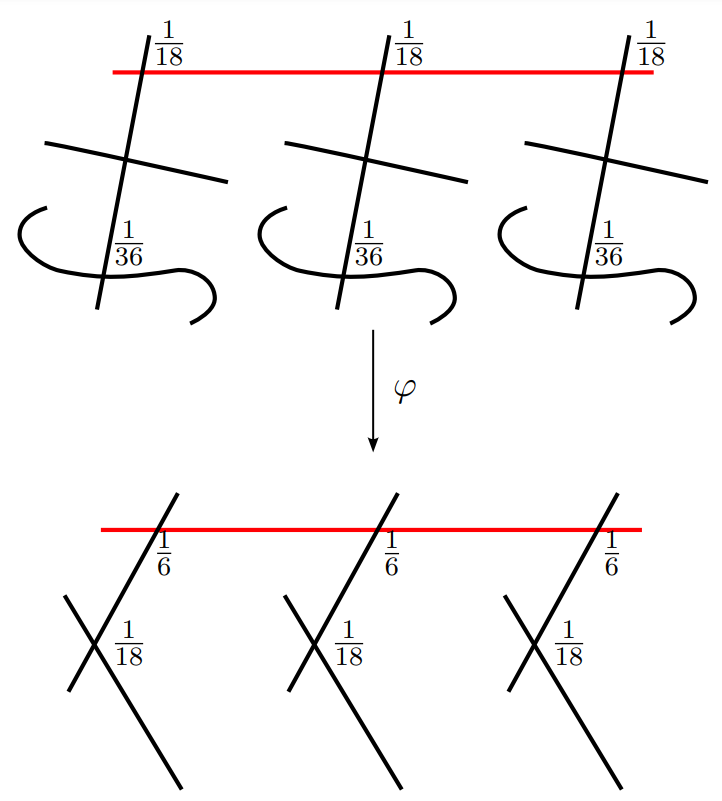}
\caption{The cover $\CY_s\to\CX_s$}
\label{fig-5}
\end{figure}

We immediately get part (a) of Theorem \ref{theresult} by contracting rational components in $\CY_s$. Part (b) follows just as in Section 2: $L/\BQ_3$ has ramification index $108$, and an anulus of width $1/18+1/36=1/12$ splits into $9$ anuli of width $1/108$.

\appendix

\section{Computer programs}

The SageMath programs in this appendix are written for SageMath Version 9.5 or newer, and require the branch \texttt{padic\_extensions} of the package mclf\footnote{Available on github: \href{https://github.com/MCLF/mclf/tree/padic_extensions/mclf}{https://github.com/MCLF/mclf/tree/padic\_extensions/mclf}}.

All Magma programs are written for Magma Version 2.27-7. They all run in under two minutes on the freely available online Magma Calculator\footnote{\href{http://magma.maths.usyd.edu.au/calc/}{http://magma.maths.usyd.edu.au/calc/}}.

Code Listings \ref{code-4} and \ref{code-5} belong to the strategy explained in Section 2, while Code Listings \ref{code-3}, \ref{code-1}, and \ref{code-2} belong to the strategy of Section 3.

In Code Listing \ref{code-4}, the minimal polynomial \texttt{H} of a certain generator $z$ of the function field of the curve $Y$ defined by Equation \eqref{eq-Aversion} is computed. It is of the form
\begin{equation}\label{eq-tameParameter}
z^3+Az^2+Bz+C,
\end{equation}
where $A,B,C$ are polynomials of degree $2$, $3$, and $4$ respectively.

Denote by $v_r$ the Gauss valuation centered at one of the roots of the polynomial $\psi$ from Equation \nolinebreak \eqref{eq-phi} of radius $r$. The valuations of the coefficients of $A,B,C$ with respect to $v_r$ are computed in Code Listing \ref{code-4} as well. They show that the Newton polygon of \texttt{H} with respect to $v_{17/12}$ is a straight line, with the valuations of $B$ and $C$ on this line and the valuation of $A$ above this line. Furthermore, $B$ is dominated by its constant coefficient and $C$ is dominated by its degree-$2$ coefficient. This shows that the residue field extension we are after is given by an equation of the form
\begin{equation*}
z^3\pm z\pm x^2=0.
\end{equation*}
(The curves defined by any choices $\pm$ are isomorphic, so this agrees with Equation \eqref{eq-char3equation}.)

Code Listing \ref{code-4} can also be used to compute the different function in the sense of \cite{ctt} of $\phi$ on the skeleton of Figure \ref{fig-6}. However, we work with valuations rather than with absolute values, and consider the different as a function on $\BP^1_{\overline{K}}$ rather than on $Y_{\overline{K}}^\textrm{an}$ (which is permissible, since $\phi$ induces a bijection above the wild topological branch locus). Thus the different function calculated below is actually $\frac{-\log_3}{3}$ applied to the different function of \cite{ctt}.

The element $z$ considered above is a tame parameter for $K(Y_{\overline{K}})$ in the sense of \cite[Section \nolinebreak 2.1.2]{ctt}, and the generator $x$ is a tame parameter for the rational function field $\overline{K}(x)$. Thus we can use \cite[Corollary 2.4.6(ii)]{ctt} to compute the different function. Let $v_r$ correspond to a point in the wild topological branch locus; then it has a unique extension $w_r$ to $K(Y_{\overline{K}})$. Differentiating \eqref{eq-tameParameter} yields
\begin{equation*}
(3z^2+2Az+B)dz=-(A'z^2+B'z+C')dx.
\end{equation*}
Denoting the different at the point corresponding to $v_r$ by $\delta(r)$, we have
\begin{IEEEeqnarray*}{rCl}
\delta(r)&=&w_r(3z^2+2Az+B)-w_r(A'z^2+B'z+C')+w_r(z)-w_r(x)\\&=&\min(1+\tfrac23v_r(C),v_r(A)+\tfrac13v_r(C),v_r(B))-\tfrac23v_r(C).
\end{IEEEeqnarray*}
Here we have used that $w_r(z)=\frac{1}{3}v_r(C)\le\frac{1}{2}v_r(B),v_r(A)$, because $w_r$ is the unique extension of $v_r$, and that
\begin{equation*}
v_r(C')+v_r(x)=v_r(C),
\end{equation*}
$C$ being by construction of the form $c_4x^4+c_2x^2+c_1x$.

The function $\delta(r)$ is positive for $\tfrac{2}{3}<r<\tfrac{17}{12}$, which explains the wild topological branch locus in Figure \ref{fig-6}.

\begin{lstlisting}[language=Magma, label={code-4}, caption=Magma program. Used to compute wild topological branch locus and induced map on special fibers]
Qp := pAdicField(3, 100);
P<t> := PolynomialRing(Qp);
K1<zeta3> := LocalField(Qp, t^2+t+1);
P<x> := PolynomialRing(K1);
K2<a> := LocalField(K1, x^9 + (9*zeta3 - 9)*x^8 + (54*zeta3 + 27)*x^7 + (54*zeta3 - 27/2)*x^6 + (243*zeta3 + 972)*x^5 + 729*zeta3*x^4 + (2916*zeta3 - 1458)*x^3 + (37179*zeta3 + 41553)*x^2 + (6561*zeta3 + 6561/8)*x - 63423*zeta3 + 155277);
R<x,y> := PolynomialRing(K2, 2);
F := (27*zeta3 - 243)*x^4 + (9*zeta3 - 18)*x^3*y + (6*zeta3 + 12)*x^2*y^2 + (-1458*zeta3 - 999)*x^3 + (-108*zeta3 - 108)*x^2*y + (36*zeta3 + 9)*x*y^2 + y^3 + (-1215*zeta3 + 1701)*x^2 + (-162*zeta3 + 81)*x*y - 27*y^2 + (1944*zeta3 + 2430)*x + (81*zeta3 + 162)*y + (729*zeta3);
P<y> := PolynomialRing(K2);
h := hom< R -> P | a, y >;
G := Factorization(h(F))[1][1];
K3<c> := LocalField(K2, G);
a := K3!a;

R<x,y> := PolynomialRing(K3, 2);
F := R!F;
H := Evaluate(F, x, x+a);
H := Evaluate(H, y, y+c);

P<u> := PolynomialRing(K3);
r := u^3 + MonomialCoefficient(H, y^2*x)*u^2 + MonomialCoefficient(H, y*x^2)*u + MonomialCoefficient(H, x^3);
K4<v> := LocalField(K3, r);
R<x,y> := PolynomialRing(K4, 2);
H := R!H;
H := Evaluate(H, y, y+v*x);

f := Valuation(K4!3);
r := 17/12;

vA0 := Valuation(MonomialCoefficient(H, y^2))/f;
vA1 := Valuation(MonomialCoefficient(H, y^2*x))/f + r;
vA2 := Valuation(MonomialCoefficient(H, y^2*x^2))/f + 2*r;
vB0 := Valuation(MonomialCoefficient(H, y))/f;
vB1 := Valuation(MonomialCoefficient(H, y*x))/f + r;
vB2 := Valuation(MonomialCoefficient(H, y*x^2))/f + 2*r;
vB3 := Valuation(MonomialCoefficient(H, y*x^3))/f + 3*r;
vC0 := Valuation(MonomialCoefficient(H, 1))/f;
vC1 := Valuation(MonomialCoefficient(H, x))/f + r;
vC2 := Valuation(MonomialCoefficient(H, x^2))/f + 2*r;
vC3 := Valuation(MonomialCoefficient(H, x^3))/f + 3*r;
vC4 := Valuation(MonomialCoefficient(H, x^4))/f + 4*r;
\end{lstlisting}

In Code Listing \ref{code-5} below we check that using the approach of Section 2, we can find a field extension $L/K$ of degree $108$ and ramification index $54$ over which $Y$ has semistable reduction.

This $L$ is the local field \texttt{K3local} in the code listing, whose ramification index is printed at the end. We construct it by taking a field extension as considered in Section 3. First, adjoin a ninth root of unity, then a zero of the discriminant of the cover considered in Section 3. Over the resulting extension, one of the disks of radius $17/12$ in Figure \ref{fig-6} contains a rational point, namely the disk corresponding to the Type II point \texttt{T.vertices()[2]}. All three disks contain a \texttt{K3local}-rational point. Finally, we check that there are rational points on $Y$ above the given rational points. They correspond to the factor of degree $1$ in the list \texttt{l}.

\begin{lstlisting}[language=Python, label={code-5}, caption=SageMath program. Used to verify that we have semistable reduction over a field extension as in Theorem \ref{theresult}]
from mclf import *

R.<t> = QQ[]
K1.<zeta9> = QQ.extension(t^6+t^3+1)
v = K1.valuation(3)
K1local = pAdicNumberField(K1, v)
R.<x> = K1[]
discoid = x^9 + (54*zeta9^3 - 27/2)*x^6 + (-243*zeta9^2 - 243*zeta9 - 243)*x^5 + (729*zeta9^2 - 729)*x^4 + (-729*zeta9^4 - 2187*zeta9^2 + 2187*zeta9 + 729)*x^3 + (10935*zeta9^5 + 4374*zeta9^4 - 2187*zeta9^3 + 2187*zeta9^2 + 2187*zeta9 + 15309)*x^2 + (13122*zeta9^5 - 6561/2*zeta9^4 - 13122*zeta9^2 + 6561/2*zeta9)*x - 13122*zeta9^5 + 6561/2*zeta9^4 - 2187*zeta9^3 + 13122*zeta9^2 - 32805*zeta9 + 24057/2
K2local  = K1local.simple_extension(discoid)
K2 = K2local.number_field()
zeta9 = K2local.embedding().approximate_generator(10)
w = K2.valuation(3)
R.<x> = K2[]
zeta3 = zeta9^3
psi = x^9 + (9*zeta3 - 9)*x^8 + (54*zeta3 + 27)*x^7 + (54*zeta3 - 27/2)*x^6 + (243*zeta3 + 972)*x^5 + 729*zeta3*x^4 + (2916*zeta3 - 1458)*x^3 + (37179*zeta3 + 41553)*x^2 + (6561*zeta3 + 6561/8)*x - 63423*zeta3 + 155277

FX.<x> = FunctionField(K2)
X = BerkovichLine(FX, w)
T = BerkovichTree(X)
T.add_point(X.gauss_point())
for xi in X.points_from_inequality(FX(psi), 45/4):
	T.add_point(xi)

alpha = -(T.vertices()[2].discoid()[0].numerator()[0])
R.<x,y> = K2[]
F = (27*zeta3 - 243)*x^4 + (9*zeta3 - 18)*x^3*y + (6*zeta3 + 12)*x^2*y^2 + (-1458*zeta3 - 999)*x^3 + (-108*zeta3 - 108)*x^2*y + (36*zeta3 + 9)*x*y^2 + y^3 + (-1215*zeta3 + 1701)*x^2 + (-162*zeta3 + 81)*x*y - 27*y^2 + (1944*zeta3 + 2430)*x + (81*zeta3 + 162)*y + (729*zeta3)
R.<y> = K2[]
F = R(F(x=alpha))
l = approximate_factorization(K2local, F)

K3local = K2local.simple_extension(T.vertices()[3].discoid()[0].numerator())
K3 = K3local.number_field()
w = K3.valuation(3)
print(w(w.uniformizer()))
\end{lstlisting}

In Code Listing \ref{code-3}, the field \texttt{K1} is a ramified extension of $\BQ_3(\zeta_9)$ of degree $2$. We compute the discriminant $\Delta$ as in Equation \eqref{eq-discriminant} and then use the functionality of mclf to compute a skeleton of the projective line separating the zeros of $\Delta$. The resulting tree \texttt{T} is used in Code Listings \ref{code-1} and \ref{code-2}.
\begin{lstlisting}[language=Python, label={code-3},caption=SageMath program. Used to compute a skeleton of the projective line separating branch points]
from mclf import *

R.<c> = QQ[]
K1.<c> = QQ.extension(c^12 - 12*c^10 + 60*c^8 - 159*c^6 + 234*c^4 - 180*c^2 + 57)
zeta9 = c^2 - 2

FX.<x> = FunctionField(K1)
R.<y> = FX[]
F = (-2*zeta9^5 + zeta9^4 + zeta9^3 - zeta9^2 + 2*zeta9 - 1)*x^4 + (zeta9^5 + zeta9^3 + zeta9^2)*x^3*y + (24*zeta9^5 - 6*zeta9^4 + 15*zeta9^3 + 12*zeta9^2 + 15)*x^3 + (9*zeta9^4)*x^2*y + y^3 + (-45*zeta9^5 + 9*zeta9^4 - 18*zeta9^3 + 45*zeta9^2 - 63*zeta9 - 36)*x^2 + (9*zeta9^5 - 27*zeta9^4 - 18*zeta9^3 + 18*zeta9^2 - 36)*x*y + (135*zeta9^5 + 135*zeta9^4 - 135*zeta9^3 - 54*zeta9^2 + 270*zeta9 - 27)*x + (9*zeta9^5 + 54*zeta9^4 + 18*zeta9^3 - 36*zeta9^2 + 27*zeta9 + 63)*y + (-189*zeta9^5 - 81*zeta9^4 + 279*zeta9^3 - 108*zeta9^2 - 270*zeta9 + 72)
FY = FX.extension(F)
v = K1.valuation(3)

X = BerkovichLine(FX, v)
T = BerkovichTree(X)
Delta = F.discriminant()
T = T.adapt_to_function(Delta)
T.permanent_completion()
\end{lstlisting}

Code Listing \ref{code-1} is analogous to Code Listing \ref{code-4}. It lets us compute the different function of the cover $\phi$ considered in Section 3, which turns out to be positive for $23/18<r<17/18$. See the explanation preceding Code Listing \ref{code-4} for an explanation of how to read off the wild topological branch locus from the output. To define the field \texttt{K2}, we use a discoid computed in Code Listing \ref{code-3}, namely the discoid corresponding to the vertex \texttt{T.vertices()[2]} of the tree \texttt{T} calculated by that code listing.

\begin{lstlisting}[language=Magma,caption= Magma program. Used for computing another wild topological branch locus, label={code-1}]
Qp := pAdicField(3, 100);
P<t> := PolynomialRing(Qp);
K1<zeta9> := LocalField(Qp, t^6+t^3+1);
P<x> := PolynomialRing(K1);

K2<a> := LocalField(K1, x^9 + (54*zeta9^3 - 27/2)*x^6 + (-243*zeta9^2 - 243*zeta9 - 243)*x^5 + (729*zeta9^2 - 729)*x^4 + (-729*zeta9^4 - 2187*zeta9^2 + 2187*zeta9 + 729)*x^3 + (10935*zeta9^5 + 4374*zeta9^4 - 2187*zeta9^3 + 2187*zeta9^2 + 2187*zeta9 + 15309)*x^2 + (13122*zeta9^5 - 6561/2*zeta9^4 - 13122*zeta9^2 + 6561/2*zeta9)*x - 13122*zeta9^5 + 6561/2*zeta9^4 - 2187*zeta9^3 + 13122*zeta9^2 - 32805*zeta9 + 24057/2);
R<x,y> := PolynomialRing(K2, 2);
F := (-2*zeta9^5 + zeta9^4 + zeta9^3 - zeta9^2 + 2*zeta9 - 1)*x^4 + (zeta9^5 + zeta9^3 + zeta9^2)*x^3*y + (24*zeta9^5 - 6*zeta9^4 + 15*zeta9^3 + 12*zeta9^2 + 15)*x^3 + (9*zeta9^4)*x^2*y + y^3 + (-45*zeta9^5 + 9*zeta9^4 - 18*zeta9^3 + 45*zeta9^2 - 63*zeta9 - 36)*x^2 + (9*zeta9^5 - 27*zeta9^4 - 18*zeta9^3 + 18*zeta9^2 - 36)*x*y + (135*zeta9^5 + 135*zeta9^4 - 135*zeta9^3 - 54*zeta9^2 + 270*zeta9 - 27)*x + (9*zeta9^5 + 54*zeta9^4 + 18*zeta9^3 - 36*zeta9^2 + 27*zeta9 + 63)*y + (-189*zeta9^5 - 81*zeta9^4 + 279*zeta9^3 - 108*zeta9^2 - 270*zeta9 + 72);
P<y> := PolynomialRing(K2);
h := hom< R -> P | a, y >;
G := Factorization(h(F))[1][1];
K3<c> := LocalField(K2, G);
a := K3!a;

R<x,y> := PolynomialRing(K3, 2);
F := R!F;
H := Evaluate(F, x, x+a);
H := Evaluate(H, y, y+c);

P<u> := PolynomialRing(K3);
r := u^3 + MonomialCoefficient(H, y^2*x)*u^2 + MonomialCoefficient(H, y*x^2)*u + MonomialCoefficient(H, x^3);
K4<v> := LocalField(K3, r);
R<x,y> := PolynomialRing(K4, 2);
H := R!H;
H := Evaluate(H, y, y+v*x);

f := Valuation(K4!3);
r := 17/18;

vA0 := Valuation(MonomialCoefficient(H, y^2))/f;
vA1 := Valuation(MonomialCoefficient(H, y^2*x))/f + r;
vB0 := Valuation(MonomialCoefficient(H, y))/f;
vB1 := Valuation(MonomialCoefficient(H, y*x))/f + r;
vB2 := Valuation(MonomialCoefficient(H, y*x^2))/f + 2*r;
vB3 := Valuation(MonomialCoefficient(H, y*x^3))/f + 3*r;
vC0 := Valuation(MonomialCoefficient(H, 1))/f;
vC1 := Valuation(MonomialCoefficient(H, x))/f + r;
vC2 := Valuation(MonomialCoefficient(H, x^2))/f + 2*r;
vC3 := Valuation(MonomialCoefficient(H, x^3))/f + 3*r;
vC4 := Valuation(MonomialCoefficient(H, x^4))/f + 4*r;
\end{lstlisting}

Code Listing \ref{code-2} computes the genus-$1$ components of the special fiber of $\CY_s$ following the strategy of Section 3. It adjoins to the completion of \texttt{K1} from Code Listing \ref{code-3} a center \texttt{a} of the discoid corresponding to \texttt{T.vertices()[2]}, where \texttt{T} is the tree computed in Code Listing \ref{code-3}. Then it uses the functionality of SageMath for computing extensions of valuations in function fields to compute an equation of the genus-$1$ curve we are interested in.

\begin{lstlisting}[language=Python, label = {code-2}, caption=SageMath program. Uses the functionality of Sage to compute another special fiber]
from mclf import *

R.<c> = QQ[]
K1.<c> = QQ.extension(c^12 - 12*c^10 + 60*c^8 - 159*c^6 + 234*c^4 - 180*c^2 + 57)
zeta9 = c^2 - 2

v = K1.valuation(3)
K1local = pAdicNumberField(K1, v)
R.<x> = K1[]
discoid = x^9 + (54*c^6 + 162)*x^6 - 243*c^4*x^5 + (729*c^4 - 729*c^2)*x^4 + (729*c^8 - 729*c^6 + 2187*c^4 + 3645*c^2 - 2187)*x^3 + (4374*c^10 + 13122*c^6 + 13122*c^4 + 13122*c^2 + 13122)*x^2 + (13122*c^10 + 6561/2*c^8 + 39366*c^4)*x + 6561*c^10 - 13122*c^8 + 15309/2*c^6 + 52488*c^4 + 6561/5*c^2
K2local  = K1local.simple_extension(discoid)
K2 = K2local.number_field()
c = K2local.embedding().approximate_generator(10)
zeta9 = c^2 - 2
w = K2.valuation(3)
R.<x> = K2[]
discoid = x^9 + (54*c^6 + 162)*x^6 - 243*c^4*x^5 + (729*c^4 - 729*c^2)*x^4 + (729*c^8 - 729*c^6 + 2187*c^4 + 3645*c^2 - 2187)*x^3 + (4374*c^10 + 13122*c^6 + 13122*c^4 + 13122*c^2 + 13122)*x^2 + (13122*c^10 + 6561/2*c^8 + 39366*c^4)*x + 6561*c^10 - 13122*c^8 + 15309/2*c^6 + 52488*c^4 + 6561/5*c^2
l = approximate_factorization(K2local, discoid)
a = -(l[0].approximate_polynomial()[0])

R.<x,y> = K2[]
F = (-2*zeta9^5 + zeta9^4 + zeta9^3 - zeta9^2 + 2*zeta9 - 1)*x^4 + (zeta9^5 + zeta9^3 + zeta9^2)*x^3*y + (24*zeta9^5 - 6*zeta9^4 + 15*zeta9^3 + 12*zeta9^2 + 15)*x^3 + (9*zeta9^4)*x^2*y + y^3 + (-45*zeta9^5 + 9*zeta9^4 - 18*zeta9^3 + 45*zeta9^2 - 63*zeta9 - 36)*x^2 + (9*zeta9^5 - 27*zeta9^4 - 18*zeta9^3 + 18*zeta9^2 - 36)*x*y + (135*zeta9^5 + 135*zeta9^4 - 135*zeta9^3 - 54*zeta9^2 + 270*zeta9 - 27)*x + (9*zeta9^5 + 54*zeta9^4 + 18*zeta9^3 - 36*zeta9^2 + 27*zeta9 + 63)*y + (-189*zeta9^5 - 81*zeta9^4 + 279*zeta9^3 - 108*zeta9^2 - 270*zeta9 + 72)
G = 0
for c in F:
	G += c[1]*K2local.approximation(c[0], 6)
FX.<x> = FunctionField(K2)
R.<y> = FX[]
G = R(G)
FY = FX.extension(G)

X = BerkovichLine(FX, w)
xi = X.point_from_discoid(x-a, 4/3)

l = xi.valuation().extensions(FY)
for vxi in l:
	print(vxi.residue_field())

\end{lstlisting}

\printbibliography

\end{document}